\documentclass[12pt,a4paper,reqno]{amsart}
\usepackage{amsthm}
\usepackage{amssymb}
\usepackage{amsmath}
\usepackage{enumitem}
\usepackage{mathtools}
\usepackage{float}
\usepackage{tikz,color,amsthm}
\usetikzlibrary{positioning}
\usetikzlibrary{chains}
\usepackage{fullpage}

\usepackage[utf8]{inputenc}

\renewcommand{\O}{\mathcal{O}}
\newcommand{\Sto}{\mathsf{Sto}}
\newcommand{\GETOUT}[1]{}

\DeclareMathAlphabet{\mathpzc}{OT1}{pzc}{m}{it}
\def\lgrey{0.9}
\def\softgrey{0.8}
\definecolor{lightgrey}{rgb}{\lgrey,\lgrey,\lgrey}
\definecolor{grey}{rgb}{\softgrey,\softgrey,\softgrey}

 \theoremstyle{definition}
 \newtheorem{theorem}{Theorem}

 \newtheorem{conjecture}[theorem]{Conjecture}
 \newtheorem{lemma}[theorem]{Lemma}

 \newtheorem{example}[theorem]{Example}
 \newtheorem{proposition}[theorem]{Proposition}
 \newtheorem{question}[theorem]{Question}
 
 \newtheorem{definition}[theorem]{Definition}

 \theoremstyle{remark}

\definecolor{mygreen}{RGB}{80,160,80}

\tikzstyle{vertex}=[circle, draw, inner sep=0pt, minimum size=6pt]
\newcommand{\vertex}{\node[vertex]}
\newcommand{\out}{\mathrm{out}}
\newcommand{\ind}{\mathrm{in}}
\newcommand{\comp}{\mathrm{comp}}

\title{A note on the lacking polynomial of the complete bipartite graph}
\author{Amal Alofi}
\author{Mark Dukes$^{\star}$}
\address{School of Mathematics and Statistics, University College Dublin, Belfield, Dublin 4, Ireland.}
\email{amal.alofi@ucdconnect.ie}
\email{mark.dukes@ucd.ie}

\begin{document}
\begingroup
\def\uppercasenonmath#1{} 
\let\MakeUppercase\relax 
\maketitle
\endgroup

\begin{abstract}
The lacking polynomial is a graph polynomial introduced by Chan, Marckert, and Selig in 2013 that is closely related to the Tutte polynomial of a graph.
It arose by way of a generalization of the Abelian sandpile model and is essentially the generating function of the level statistic on the set of recurrent configurations, called stochastically recurrent states, for that model.
In this note we consider the lacking polynomial of the complete bipartite graph. 
We classify the stochastically recurrent states of the stochastic sandpile model on the complete bipartite graphs $K_{2,n}$ and $K_{m,2}$ where the sink is always an element of the set counted by the first index.
We use these characterizations to give explicit formulae for the lacking polynomials of these graphs. 
Log-concavity of the sequence of coefficients of these two lacking polynomials is proven, and we conjecture log-concavity holds for this general class of graphs. 
\end{abstract}
\bigskip

Graph polynomials provide a useful means by which to capture graph invariants.
The most studied graph polynomial is arguably the Tutte polynomial.
This bivariate polynomial $T_G(x,y)$ contains as specializations the chromatic polynomial, 
the Jones polynomial, and the reliability polynomial.
In addition to this, the evaluation of the Tutte polynomial at particular integral points close to the origin corresponds to the enumeration of combinatorially significant substructures in the graph.
The Tutte polynomial $T_M(x,y)$ of a matroid shares similarly rich properties.

In combinatorics, the property of number sequences being logarithmically concave (or log-concave) has attracted considerable attention over the years, 
see Stanley~\cite{stanley} for an overview of such results.
Recent advances in this topic include the proof of the Heron-Rota-Welsh conjecture for matroids by Adiprasito, Huh, and Katz~\cite{ahk} that concerns the sequence of coefficients of the characteristic polynomial of a matroid.
Quite a few specializations of the Tutte polynomial have been conjectured to be log-concave, and some of these have been established.  
For instance, Eur, Huh and Larson~\cite[Cor. 1.14]{ehl} proved the following result:

\medskip
\begin{center}
\begin{minipage}{0.8\textwidth}
{\it{For any rank $r$ matroid $M$, the coefficients of the polynomial $q^{r}T_M(q^{-1}, 1+q)$ form a log-concave sequence with no internal zeroes.}}
\end{minipage}
\end{center}
\medskip

This result establishes log-concavity for the special case of when $M$ is the cycle matroid of a graph $M=M(G)$.
An expression for the Tutte polynomial of the complete bipartite graph in terms of an exponential generating function of an associated polynomial was given by
Martin and Reiner~\cite[Prop. 7]{martinreiner} , while a closed form for the Tutte polynomial $T_{K_{n,2}}(x,y)$ is given in \cite[Thm 4.1]{allagan}:
$$T_{K_{n,2}}(x,y) = x^2(x+1)^{n-1} + \sum_{i=1}^{n-1} (x+y)^i (x+1)^{n-i-1}.$$
Indeed it is not clear that this form can be used to establish log-concavity of the sequence of coefficients
of $q^{r(M(K_{n,2}))}T_{M(K_{n,2})}(q^{-1}, 1+q)$.

Our interest in this note is a newly developed graph polynomial that is termed the {\em lacking polynomial} of a graph.
The lacking polynomial is a graph polynomial introduced by Chan, Marckert, and Selig~\cite{ssm} in 2013 that is closely related to the Tutte polynomial of a graph.
It arose from a variant of the Abelian sandpile model called the {\em stochastic sandpile model (SSM)}.
The lacking polynomial is the generating function of the level statistic on the set of recurrent configurations, called stochastically recurrent states, for that new model.
In \cite{ssm} it was shown that the lacking polynomial satisfies a deletion-contraction recurrence similar to that of the Tutte polynomial.
This then motivates the question as to whether the lacking polynomial also satisfies the property of being log-concave in quite general instances.

In this note we consider the lacking polynomial of the complete bipartite graph.
We classify the stochastically recurrent states of the stochastic sandpile model on the complete bipartite graphs $K_{2,n}$ and $K_{m,2}$ where the sink is always an element of the set counted by the first index.
We use these characterizations to give explicit formulae for the lacking polynomials of these graphs.
Log-concavity of the sequence of coefficients of these two lacking polynomials is proven, and we conjecture log-concavity holds for this general class of graphs.

The stochastically recurrent configurations of the SSM are derivable from considering all orientations of the underlying graph, and this naturally leads to the topic of tournaments and score sequences.
Recent related work includes Claesson, Dukes, Frankl\'in, and Stef\'ansson's paper~\cite{cdfs},
Selig's paper studying the SSM on the complete graph \cite{selig1} and Selig and Zhu's paper on the SSM for the complete bipartite graph~\cite{seligcompletebipartite}.
In addition to this, it also adds to the growing literature on graph polynomials related to the sandpile model \cite{graphpolys}.

In the stochastic sandpile model the classical sandpile toppling rule is replaced with an alternative rule whereby,
on toppling an unstable vertex, a grain may (but does not have to) be sent to each neighbouring vertex. 
It can be viewed as a Markov chain on the set of non-negative configurations where at each time step a grain is added to a random node followed by a complete stochastic stabilization.

One consequence of this is that the toppling of a vertex does not necessarily result in it becoming stable.
As part of their paper the authors introduced a notion of stochastically recurrent states as the analog of recurrent states for the classical model.
They provided a characterization for such states in terms of {\it orientations compatible with configurations}.
Further research into this model can be found in the papers~\cite{selig1,selig2}.

Let $G$ be a 
finite, unoriented, connected and loop-free graph with a distinguished vertex, $s$, called the sink.
A {\em stable configuration} on $G$ is an assignment of non-negative integers to each non-sink vertex such that the number of grains at a given vertex is strictly
less than its degree.
First, we will recall a definition from Chan et al.~\cite{ssm} which explains what it means for an orientation to be compatible with a configuration.
\begin{definition}[Chan et al.~\cite{ssm}]\label{defcom}
Let $c$ be a configuration on $G$. 
An orientation $\O$ of $G$ is an assignment of a direction to each of the edges of $G$.
We say that configuration {\it{$c$ is compatible with $\O$}} (and likewise {\it{$\O$ is compatible with $c$}}) if for all non-sink vertices $v$ in $G$,
$$\ind_{\O} (v) ~ \geq ~ d(v)-c_{i},$$
where $\ind_{\O} (v)$ is the number of incoming edges to $v$ w.r.t. $\O$ and $d(v)$ is the degree of vertex $v$.
In other words, the number of incoming edges and grains on a vertex reaches at least the degree of this vertex.
\end{definition}
We denote by $\comp(\O)$ the set of stable configurations on $G$ that are compatible with $\O$.
Note that, in comparison to the paper~\cite{ssm}, the inequality in Definition~\ref{defcom} is missing one on the right hand side.
This is because of a subtle change to the model. 
In \cite{ssm} they considered a vertex unstable if the number of grains at a vertex exceeds its degree, whereas we consider a vertex unstable if the number of grains at vertex is not less than the degree which is in line with the definition of the ASM in \cite{dlb}. The two notions are equivalent.

\begin{theorem}[Chan et al.~\cite{ssm}]\label{thcom}
A stable configuration $c$ on $G$ is {\em stochastically recurrent} if and only if there exists an orientation $\O$ on $G$ such that $c\in \comp(\O)$. It implies that $\Sto(G)$, the set of all  stochastically recurrent states on $G$, may be written
$$\Sto(G) = \bigcup_{ \mathrm{orientations }~ \O \atop \mathrm{of } ~ G } \ \comp(\O)$$
where the union is taken over all orientations on $G$.
\end{theorem}

Chan et al.~\cite{ssm} also introduced the lacking polynomial of a graph to be the generating function counting the stochastically recurrent configurations according to the number of grains by which a given configuration differs from the maximally stable configuration. 

\begin{definition}[Chan et al.~\cite{ssm}]\label{dela}
The lacking polynomial $L_G(x)$ is 
$$L_G (x) :=\sum_{c\in \Sto(G)}  x^{\ell(c)},$$ where
$$\ell(c) :=\sum_{v\in V(G)\setminus \{s\}}  l(v)$$
and $l(v) := d(v)- c(v)-1$ is the {\em lacking number} of vertex $v$.
\end{definition}

Readers familiar with the sandpile model literature will notice that the lacking polynomial is essentially the level polynomial (see e.g. Cori and Le Borgne~\cite{clb}) of the graph over the set of stochastically recurrent states with the sequence of coefficients reversed.

In this note we consider the stochastic sandpile model on the complete bipartite graph 
$K_{m,n}$ with vertex set $\{v_0,v_1,\ldots,v_{m+n-1}\}$.
We will treat $v_0$ as the sink $s$ and this graph has edges connecting vertices in 
the sets $\{v_0, v_1, v_2,\ldots,v_{m-1}\}$ and $\{v_m, v_{m+1}, \ldots,v_{m+n-1}\}$.

We characterise stochastically recurrent states on the graphs $K_{2,n}$ and $K_{m,2}$ and use these
characterisations to give expressions for the lacking polynomials $L_{2,n}(x)$ and $L_{m,2}(x)$ on those graphs. 
We also prove that the sequence of coefficients of both $L_{2,n}(x)$ and $L_{m,2}(x)$ are log-concave. 
This note is motivated by Alofi and Dukes~\cite{alofidukes} that considers rectangular tableaux representations of recurrent states of the Abelian sandpile model on the complete bipartite graph, and transformations upon them.

Theorem~\ref{thcom} allows us to write an expression for stochastically recurrent states on the complete bipartite graph$K_{m,n}$:

\begin{proposition}\label{thSto}
The set of stochastically recurrent states of the stochastic sandpile model on $K_{m,n}$ is
$$\Sto (K_{m,n}) = \bigcup_{ \mathrm{orientations }~ \O \atop \mathrm{of }~ K_{m,n}} \ \left\{(c_{1},\ldots,c_{m+n-1}) ~:~ \out_{\O} (v_i)\leq c_{i} < d(v_{i}), ~ \forall 1\leq i \leq m+n-1\right\}.$$
\end{proposition}

\begin{proof}
Let  $c=(c_1, c_2,\ldots, c_{m+n-1})$ be a stable configuration on $K_{m,n}$.
Suppose it is compatible with an orientation $\O$ where $c_{i}< d(v_i)$  for all $1\leq i \leq m+n-1$. 
According to the Definition~\ref{defcom} we must have:
$$\ind_{\O} (v_i)\geq d(v_{i})-c_{i}.$$
This means
$c_{i} ~\geq~ d(v_{i})-\ind_{\O} (v_i) ~=~  \out_{\O} (v_i).$
When we combine the application of Definition~\ref{defcom} with Theorem~\ref{thcom} for $G=K_{m,n}$ we find 
$$\Sto (K_{m,n}) = \bigcup_{ \mathrm{orientations }~ \O \atop \mathrm{of }~ K_{m,n}} \ \{(c_{1},\ldots,c_{m+n-1}) ~:~ \out_{\O} (v_i)\leq c_{i}< d(v_{i}), ~ \forall 1\leq i \leq m+n-1\}.$$
\end{proof}

\begin{example}
Consider the graph $K_{2,2}$. 
To determine the stochastically recurrent states compatible with each orientation $\O$ of graph $K_{2,2}$ first we find $\out_{\O}(v_i)$ for all $1\leq i\leq 3$, and then when we apply Prop.~\ref{thSto}.
See Table~\ref{all:k:twotwo} for a listing of the configurations that are compatible with each orientation.
\begin{table}[h!]
\begin{center}
\def\fgfg{0.6}
\begin{tabular}{|l|@{$\qquad$}l|} \hline
\begin{minipage}[b]{1.2cm}
\begin{tikzpicture}[scale=0.6,node distance={6mm}, thick, main/.style = {draw}]
	\vertex[fill] (v0) at (0,1) [label=above:$v_{0}$] {};
	\vertex[fill] (v1) at (1,1) [label=above:$v_{1}$] {};
	\vertex[fill] (v2) at (0,0) [label=below:$v_{2}$] {};
	\vertex[fill] (v3) at (1,0) [label=below:$v_{3}$] {};
		\draw [<-] (v0) -- (v2);
		\draw[->] (v0) -- (v3);
		\draw[->] (v1) -- (v2);
		\draw[<-] (v1) -- (v3);
\end{tikzpicture}
\end{minipage}
& 
\begin{minipage}[b]{\fgfg\textwidth}
$\out_{\O} (v_1)=1, \out_{\O} (v_2)=1, \out_{\O} (v_3)=1$.\\ 
$\implies$ $1\leq  c_1 < 2$, $1\leq  c_2  < 2$, and $1\leq  c_3 < 2$.\\
The contribution to $\Sto(K_{2,2})$ is $\{(1,1,1)\}.$\\
\end{minipage} \\ \hline
\begin{minipage}[b]{1cm}
\begin{tikzpicture}[scale=0.6,node distance={6mm}, thick, main/.style = {draw}]
	\vertex[fill] (v0) at (0,1) [label=above:$v_{0}$] {};
	\vertex[fill] (v1) at (1,1) [label=above:$v_{1}$] {};
\vertex[fill] (v2) at (0,0) [label=below:$v_{2}$] {};
	\vertex[fill] (v3) at (1,0) [label=below:$v_{3}$] {};
	\draw [->] (v0) -- (v2);
		\draw[<-] (v0) -- (v3);
\draw[<-] (v1) -- (v2);
		\draw[->] (v1) -- (v3);
\end{tikzpicture}
\end{minipage}
&
\begin{minipage}[b]{\fgfg\textwidth}
$\out_{\O} (v_1)=1, \out_{\O} (v_2)=1, \out_{\O} (v_3)=1$.\\ 
$\implies$ $1\leq  c_1 < 2$, $1\leq  c_2  < 2$, and $1\leq  c_3 < 2$.\\
The contribution to $\Sto(K_{2,2})$ is $\{(1,1,1)\}.$\\
\end{minipage} \\ \hline
\begin{minipage}[b]{1cm}
\begin{tikzpicture}[scale=0.6,node distance={6mm}, thick, main/.style = {draw}]
	\vertex[fill] (v0) at (0,1) [label=above:$v_{0}$] {};
	\vertex[fill] (v1) at (1,1) [label=above:$v_{1}$] {};
\vertex[fill] (v2) at (0,0) [label=below:$v_{2}$] {};
	\vertex[fill] (v3) at (1,0) [label=below:$v_{3}$] {};
	\draw [->] (v0) -- (v2);
		\draw[->] (v0) -- (v3);
\draw[->] (v1) -- (v2);
		\draw[<-] (v1) -- (v3);
\end{tikzpicture}
\end{minipage}
&
\begin{minipage}[b]{\fgfg\textwidth}
$\out_{\O} (v_1)=1, \out_{\O} (v_2)=0, \out_{\O} (v_3)=1.$\\ 
$\implies$ $1\leq  c_1 < 2$, $0\leq  c_2  < 2$, and $1\leq  c_3 < 2$.\\
The contribution to $\Sto(K_{2,2})$ is $\{(1,0,1),(1,1,1)\}.$\\
\end{minipage} \\ \hline
\begin{minipage}[b]{1cm}
\begin{tikzpicture}[scale=0.6,node distance={6mm}, thick, main/.style = {draw}]
	\vertex[fill] (v0) at (0,1) [label=above:$v_{0}$] {};
	\vertex[fill] (v1) at (1,1) [label=above:$v_{1}$] {};
\vertex[fill] (v2) at (0,0) [label=below:$v_{2}$] {};
	\vertex[fill] (v3) at (1,0) [label=below:$v_{3}$] {};
	\draw [->] (v0) -- (v2);
		\draw[->] (v0) -- (v3);
\draw[<-] (v1) -- (v2);
		\draw[->] (v1) -- (v3);
\end{tikzpicture}
\end{minipage}
&
\begin{minipage}[b]{\fgfg\textwidth}
$\out_{\O} (v_1)=1, \out_{\O} (v_2)=1, \out_{\O} (v_3)=0.$\\ 
$\implies$ $1\leq  c_1 < 2$, $1\leq  c_2  < 2$, and $0\leq  c_3 < 2$.\\
The contribution to $\Sto(K_{2,2})$ is $\{(1,1,0),(1,1,1)\}.$\\
\end{minipage} \\ \hline
\begin{minipage}[b]{1cm}
\begin{tikzpicture}[scale=0.6,node distance={6mm}, thick, main/.style = {draw}]
	\vertex[fill] (v0) at (0,1) [label=above:$v_{0}$] {};
	\vertex[fill] (v1) at (1,1) [label=above:$v_{1}$] {};
	\vertex[fill] (v2) at (0,0) [label=below:$v_{2}$] {};
	\vertex[fill] (v3) at (1,0) [label=below:$v_{3}$] {};
	\draw [->] (v0) -- (v2);
		\draw[->] (v0) -- (v3);
\draw[<-] (v1) -- (v2);
		\draw[<-] (v1) -- (v3);
\end{tikzpicture}
\end{minipage}
&
\begin{minipage}[b]{\fgfg\textwidth}
$\out_{\O} (v_1)=0, \out_{\O} (v_2)=1, \out_{\O} (v_3)=1.$\\  
$\implies$ $0\leq  c_1 < 2$, $1\leq  c_2  < 2$, and $1\leq  c_3 < 2$.\\
The contribution to $\Sto(K_{2,2})$ is $\{(0,1,1),(1,1,1)\}.$\\
\end{minipage} \\ \hline
\end{tabular}
\end{center}
\caption{Checking all orientations of $K_{2,2}$.\label{all:k:twotwo}}
\end{table}
It follows that
$$\Sto(K_{2,2})=\{(0,1,1),(1,0,1),(1,1,0),(1,1,1)\}.$$
\end{example}

We can provide crude lower and upper bounds on the number of stochastically recurrent configurations by making use of the fact~\cite[Prop. 2.3]{ssm}
that $\Sto(K_{m,n})$ properly contains the set of classically recurrent states, and such states are in 1-1 correspondence with 
the number of spanning trees of the underlying graph.
Fieldler and Sedlacek~\cite{spanningkmn} showed the number of spanning trees of the complete bipartite graph $K_{m,n}$ is $n^{m-1}m^{n-1}$. 
Thus the number of stochastically recurrent configurations on $K_{m,n}$ is at least $ n^{m-1}m^{n-1}.$
A trivial upper bound is achieved by noting that the stochastically recurrent states are stable configurations, of which there are $n^{m-1}m^{n}$ many. Thus
\begin{align} \label{thmanup}
n^{m-1}m^{n-1} ~\leq~ |\Sto(K_{m,n})| ~\leq~ n^{m-1}m^{n}.
\end{align}
Note that the upper bound differs from the lower bound only by a factor of $m$.

\begin{question}
Can it be determined whether or not the number of stochastically recurrent states dominates the set of stable states? I.e. can it be decided
$$|\Sto(K_{m,n})| \lessgtr \dfrac{n^{m-1}m^{n}}{2} ?$$
\end{question}

The set of stable configurations on $K_{2,n}$ is $$\{(c_1,c_2,\ldots,c_{n+1})~:~0\leq c_1<n \mbox{ and } c_i \in \{0,1\},~  \forall \ 2\leq i \leq n+1\}.$$
In order to calculate the lacking polynomial $L_{2,n} (x)$ we first need an explicit characterization of the set $\Sto(K_{2,n})$ without repetitions. 
\begin{proposition}\label{firstchar}
$$\Sto(K_{2,n}) = \{(c_1,c_2,\ldots,c_{n+1}) : c_2,\ldots,c_{n+1} \in \{0,1\} \mbox{ and } c_1 \geq |{j \in [2,n+1]: c_j=0}| \}.$$
\end{proposition}
\begin{proof}
Using Proposition~\ref{thSto} we know 
$$\Sto(K_{2,n})=\{c=(c_1,\ldots,c_{n+1}) ~:~ \exists \mbox{ an orientation $\O$ of $K_{2,n}$ compatible with $c$}\}.$$
Let us suppose $c=(c_1,\ldots,c_{n+1})$ is a member of $\Sto(K_{2,n})$. 
This means $$n-\ind_{\O}(v_1)=\out_{\O}(v_1) \leq c_1<n$$ and 
$$2-\ind_{\O}(v_j)  =  \out_{\O}(v_j)  \leq  c_j  <  2 \mbox{ for all }2\leq j \leq n+1.$$ 
Let us now see under what conditions one can construct an orientation $\O$ on $K_{2,n}$ that is compatible with a given $c$. 
Let $X$ be the set of indices $j$ in $[2,n+1]$ for which $c_j=0$, and where we use the notation $[a,b]:=\{a,a+1,\ldots,b\}$.
For any $j$ in $X$ the number of outgoing edges at vertex $v_j$ is zero $(\out_{\O}(v_j)=0)$, because we know that $\out_{\O}(v_j)\leq c_j<2$, so if $c_j=0$ then $\out_{\O}(v_j)=0$. 
Therefore, for all $j$ in $X$ there is one outgoing edge from $v_1$ to $v_j$, and hence vertex $v_j$ has one incoming edge from $v_1$. 
So the number of outgoing edges from $v_1$ is greater than or equal to the number of elements in $X$.

For any $j$ in $[2,n+1] \backslash  X$, we know that there is at most one outgoing edge from $v_j$ as $\out_{\O}(v_j)\leq c_j<2$.
So if $c_j=1$ then $\out_{\O}(v_j)\leq 1$. 
With these considerations in mind, we can construct an orientation $\O$ on $K_{2,n}$ that is compatible with a stable configuration $c$:
\begin{enumerate}
\item[(i)]  If $c_j=1$ with $2\leq j\leq n+1$ then there is at most one outgoing edge from $v_j.$
\item[(ii)] If $c_j=0$ and $2\leq j\leq n+1$ then there are no outgoing edges from $v_j.$ 
\item[(iii)] The number of outgoing edges from vertex $v_1$  is greater than or equal the number of vertices $v_j$ when $c_j=0$ for all $2\leq j\leq n+1.$
We know that the $\out_{\O}(v_1)$ is less than or equal to $c_1$ and greater than or equal to the number of vertices $v_j$ when $c_j=0$ for all $2\leq j\leq n+1.$ Therefore $c_1$ is greater than or equal to the number of vertices $v_j$ when $c_j=0$ for all $2\leq j\leq n+1.$
\end{enumerate}
There are no other restrictions that forbid us from constructing such an orientation $\O$. 
Therefore we can write down the following self-contained expression for the set $\Sto(K_{2,n})$ that does not depend on an orientation $\O$:
$$\Sto(K_{2,n}) = \{(c_1,c_2,\ldots,c_{n+1}) : c_2,\ldots,c_{n+1} \in \{0,1\} \mbox{ and } c_1 \geq |{j \in [2,n+1]: c_j=0}| \}.$$
\end{proof}

\begin{theorem}\label{thmla}
The lacking polynomial of the graph $K_{2,n}$ is 
$$L_{2,n} (x) =\sum_{k=0}^{n-1} \sum_{i=0}^{k} \binom{n}{i} x^k.$$
\end{theorem}

\begin{proof}
Definition~\ref{dela} gives
$$L_{2,n}(x)=\sum_{c \in \Sto(K_{2,n})} x^{\ell(c)}$$ where 
$$\ell(c) :=\sum_{v_i\in V(K_{2,n})\setminus \{v_0\}}  l(v_i) \qquad \mbox{ and }\qquad l(v_i):=d(v_i)- c(v_i)-1.$$
Proposition~\ref{firstchar} provides an explicit expression for the set $\Sto(K_{2,n})$ that can be used to calculate the lacking polynomial.
Let $c$ be in $\Sto(K_{2,n})$. 
Let $l_1=n-c_1-1$ be the lacking number at vertex $v_1$, and let $l_j=1-c_j$ be the lacking number at vertices $v_2,\ldots,v_{n+1}$ for all $j\in[2,n+1]$, so $l_j=0$ when $c_j=1$ and $l_j=1$ when $c_j=0.$
Let $i$ be the number of vertices $v_j$ with ${j \in [2,n+1]}$ and $c_j=0$ for which $l_j=1$.
The remaining vertices have $l_j=0$.  
Then $x^{\ell(c)}$ factors as $x^{\ell(c)} = x^{i}x^{l_1}.$
Since $c_1 \geq |{j \in [2,n+1]: c_j=0}|$, we conclude that $c_1\geq i$, therefore the lacking number at $v_1$ will be between 0 and $n-1-i$.
Since there are ${\tbinom {n}{i}}$ combinations of $i$ vertices with lacking number 1 among $\{v_2,\ldots,v_{n+1}\}$, we obtain 
$$L_{2,n}(x)=\sum_{i=0}^{n}\binom{n}{i}x^i \sum_{l_1=0}^{n-1-i}x^{l_1}.$$
Now setting $k=i+l_1$ we have
$$L_{2,n}(x)~=~\sum_{k=0}^{n-1}\sum_{i=0}^{k}\binom{n}{i} x^k.
$$ 
\end{proof}
Note that $L_{2,n}(1) = n2^{n-1} =n^{2-1}2^{n-1}$, and this matches the lower bound stated in (\ref{thmanup}).

A configuration $c=(c_1,c_2,\ldots,c_{m+1})$ on $K_{m,2}$ is stable precisely when 
$c_i \in \{0,1\}$ for all  $i\in\{1,\ldots,m-1\}$ and $0\leq c_j<m$ for all $j\in\{m,m+1\}$.
To calculate the lacking polynomial $L_{m,2}(x)$ requires an explicit characterization of the set $\Sto(K_{m,2})$. 

\begin{proposition}\label{secondchar}
$$\Sto(K_{m,2}) = \{(c_1,c_2,\ldots,c_{m+1}) ~:~ c_1,\ldots,c_{m-1} \in \{0,1\} \mbox{ and } c_{m}+c_{m+1} \geq m-1+|X| \}.$$
\end{proposition}
\begin{proof}
From Proposition~\ref{thSto} we have 
$$\Sto(K_{m,2})=\{c=(c_1,\ldots,c_{m+1}) ~:~ \exists \mbox{ orientation $\O$ of $K_{m,2}$ compatible with $c$}\}.$$
Suppose $c=(c_1,\ldots,c_{m+1})$ is a member of $\Sto(K_{m,2}).$ 
This means that $$n-\ind_{\O}(v_i)=\out_{\O}(v_i) \leq c_i<2 \mbox{ for all }i\in\{1,\ldots,m-1\}$$ and 
$$2-\ind_{\O}(v_j) =\out_{\O}(v_j) \leq c_j < m \mbox{ for all }   j\in\{m,m+1\}.$$ 
Let us see under what conditions one can construct an orientation $\O$ on $K_{m,2}$ that is compatible with a given $c$. 
Let $X$ be the set of indices $i$ in $[1,m-1]$ for which $c_i=0$.
For any $i\in X$ the number of outgoing edges at vertex $v_i$ is zero $(\out_{\O}(v_i)=0)$ since $\out_{\O}(v_i)\leq c_i<2$. 
So if $c_i=0$ then $\out_{\O}(v_i)=0$.

Therefore, for all $i$ in $X$ there are is one outgoing edge from each of $v_m$ and $v_{m+1}$ to $v_i$.
Moreover, vertex $v_i$ has one incoming edge from each of $v_m$ and $v_m+1$. 
So the number of outgoing edges from $v_m$ is greater than or equal to the number of elements in $X$.
Also, the number of outgoing edges from $v_{m+1}$ is greater than or equal to the number of elements in $X$. 
Therefore $c_m\geq |X|$ and $c_{m+1}\geq |X|$.

For any $i\in [1,m-1]\backslash X$, we know that there is at most one outgoing edge from $v_i$ since $\out_{\O}(v_j)\leq c_j<2$, so if $c_j=1$ then $\out_{\O}(v_j)\leq 1.$ 
So the total number of outgoing edges from $v_m$ and $v_{m+1}$ to $v_i$ must at least equal $m-1-|X|.$ 
Therefore we can then construct such an orientation $\O$ of $K_{m,2}$ that is compatible with a stable configuration $c$ precisely when $c_m+c_{m+1}\geq 2|X|+m-1-|X|=m-1+|X|.$
There are no other restrictions that forbid us from constructing such an orientation $\O$. 
Therefore we can write down the following self-contained expression for the set $\Sto(K_{m,2})$ that does not depend on an orientation $\O$:
$$\Sto(K_{m,2}) = \{(c_1,c_2,\ldots,c_{m+1}) ~:~ c_1,\ldots,c_{m-1} \in \{0,1\} \mbox{ and } c_{m}+c_{m+1} \geq m-1+|X| \}.$$
\end{proof}

\begin{theorem}\label{thmll}
The lacking polynomial for the graph $K_{m,2}$ is 
$$L_{m,2} (x) =\sum_{k=0}^{m-1}  S(m-1,k) x^k$$ where
$$S(m-1,k)=\sum_{q=0}^{k} \sum_{r=0}^{q} \binom{m-1}{r} \mbox{ for all }\ 0\leq k\leq m-1.$$
\end{theorem}

\begin{proof}
The lacking polynomial $L_{m,2}(x)$, given in Definition~\ref{dela}, is
$$L_{m,2}(x) = \sum_{c \in \Sto(K_{m,2})} x^{\ell(c)}$$ where 
$$\ell(c) =\sum_{v_i\in V(K_{m,2})\setminus \{v_0\}}  l(v_i), \qquad \mbox{ and } \qquad l(v_i)=d(v_i)- c(v_i)-1.$$ 
Proposition~\ref{secondchar} provides 
an explicit expression for the set $\Sto(K_{m,2})$ that can now be used to calculate $L_{m,2}(x)$.
Let $c$ be in $\Sto(K_{m,2})$. 
Let $l_i=1-c_i$ be the lacking number of vertex $v_i$ for all $i \in [1,m-1]$, so that $l_i=0$ when $c_i=1$ and $l_i=1$ when $c_i=0.$
For the vertices $v_m$ and $v_{m+1}$ the lacking numbers are $l_m=m-1-c_m$ and $l_{m+1}=m-1-c_{m+1}$, respectively. 
We factor $x^{\ell(c)} = x^{|X|}x^{l_m+l_{m+1}}$ and
can now write $c_{m}+c_{m+1} \geq m-1+|X|$ in terms of lacking numbers as:
$$m-1-l_{m}+m-1-l_{m+1} \geq m-1+|X|$$ which is equivalent to $$m-1-|X|\geq l_m +l_{m+1}.$$
Now suppose that $r = l_m + l_{m+1}$, then $r$ ranges in between 0 and $m-1 -|X|$.
For each choice of $r$ we have exactly $r + 1$ choices for $(l_m, l_{m+1})$.
Now let $j=|X|$ be the number of vertices $v_i$ with ${i \in [1,m-1]}$ and $c_i=0$ for which then $l_i=1$. 
The remaining vertices will have $l_i=0$. 
So there are ${\tbinom {m-1}{j}}$ combinations of $j$ vertices with lacking number 1 among $\{v_1,\ldots,v_{m-1}\}$, and we obtain 
$$L_{m,2}(x)=\sum_{j=0}^{m-1}\binom {m-1}{j}x^j \sum_{r=0}^{m-1-j}(r+1) x^{r}= \sum_{j=0}^{m-1}\sum_{r=0}^{m-1-j}\binom {m-1}{j}(r+1) x^{r+j}.$$
Set $k=j+r$. Then $k$ runs from 0 to $m-1$ so that $j$ can run from 0 to $k$ and, with $r = k-j$, we obtain the sum
$$L_{m,2}(x)=\sum_{k=0}^{m-1}\sum_{j=0}^{k}\tbinom {m-1}{j}(k-j+1) x^{k}.$$
This is equal to
$$L_{m,2}(x)=\sum_{k=0}^{m-1}  S(m-1,k) x^k$$ where
$$S(m-1,k)=\sum_{q=0}^{k} \sum_{r=0}^{q} \binom{m-1}{r} \mbox{ for all }  0\leq k\leq m-1.\qedhere$$
\end{proof}
Note that the sequence $(L_{m,2}(1)_{m\geq 1}$ corresponds to sequence A084851 in the OEIS~\cite{oeis}.

A sequence $a_0,a_1,\ldots,a_n$ of non-negative real numbers is said to be {\em logarithmically concave}, or {\em log-concave}, if for all $0 < k < n$,
$a_k^2- a_{k-1} a_{k+1}\geq 0.$

\begin{lemma}\label{thmlog}
Suppose $(x_0,x_1,\ldots,x_n)$ is a log-concave sequence.
Then the sequence $(z_0,\ldots,z_n)$ of partial sums defined by
$$z_k=\sum_{i=0}^{k} x_i $$
is also log-concave. 
\end{lemma}

\begin{proof}
Wang and Yeh~\cite{wang}
 proved that 
if sequences $(x_k)$ and $(y_k)$ are log-concave, then the sequence $(z_k)$ where $z_k$ is the ordinary convolution 
$$z_k=\sum_{i=0}^{k} x_i y_{n-i}$$
is log-concave.
The sequence $(y_0,\ldots,y_k)=(1,\ldots,1)$ is trivially log-concave. 
The statement of the lemma follows since it is the special case with all $y_j$'s replaced with 1s.
\end{proof}

\begin{theorem}\label{thmlogc}
\begin{enumerate}
\item[]
\item[(i)] The sequence of coefficients of the lacking polynomial $L_{2,n}(x)$ is log-concave.
\item[(ii)] The sequence of coefficients of the lacking polynomial $L_{m,2}(x)$ is log-concave.
\end{enumerate}
\end{theorem}

\begin{proof}
(i)
From Theorem~\ref{thmla} we have
$$L_{2,n} (x) =\sum_{k=0}^{n-1}  T(n,k) x^k \quad \mbox{ where }\quad T(n,k):=\sum_{i=0}^{k} \binom{n}{i}.$$
It is well-known that the sequence of binomial coefficients
$\left(\binom{n}{k}\right)_{k=0,1,2,\ldots,n}$
is log-concave. By Lemma~\ref{thmlog}, the sequence $(T(n,k))_{k=0,\ldots,n}$ 
is log-concave. 
Therefore the sequence of coefficients of the lacking polynomial $L_{2,n}(x)$ is log-concave.
\ \\ \noindent
(ii)
From Theorem~\ref{thmll} we have
$$L_{m,2} (x) =\sum_{k=0}^{m-1}  S(m-1,k) x^k$$ where
$$S(m-1,k):=\sum_{q=0}^{k} \sum_{r=0}^{q} \binom{m-1}{r} \mbox{  for all } 0<k\leq m-1.$$
Let $R(m-1,q) := \sum_{r=0}^{q} \binom{m-1}{r}$.
By Lemma~\ref{thmlog} the sequence $(R(m-1,q))_{q=0,\ldots,m-1}$
is log-concave. Then, again by an application of Lemma~\ref{thmlog}, the sequence
$\left(S(m-1,k)\right)_{k=0,\ldots,m-1}$ $ =$ $\left( \sum_{q=0}^k R(m-1,q) \right)_{k=0,\ldots,m-1}$
is also log-concave. 
Therefore the sequence of coefficients of the lacking polynomial $L_{m,2}(x)$ is log-concave.
\end{proof}

The two results in Theorem~\ref{thmlogc} suggest that log-concavity might be a property of these lacking polynomials for the general complete bipartite graph.
We have also verified this for all $(m,n)$ with $m,n\geq 2$ and $m+n\leq 8$ (see Table~\ref{tabtwo}).
Since log-concavity of a sequence implies unimodality of the sequence, we may also conjecture this latter property in the event that log-cavity does not hold.

\protect{
\begin{conjecture}\label{conjeleven}
Let $m,n\geq 2$.
\begin{enumerate}
\item[(i)] The sequence of coefficients of $L_{m,n}(x)$ is log-concave.
\item[(ii)] The sequence of coefficients of $L_{m,n}(x)$ is unimodal.
\end{enumerate}
\end{conjecture}
}

It is worth mentioning that the more general property of these polynomials having roots that all lie in the region $S=\{z \in \mathbb{C}~:~ 2\pi/3 < \arg(z) < 4\pi/3\}$ 
, which would imply log-concavity by Stanley~\cite[Prop. 7]{stanley}, does not hold.
This is verified by considering the roots of $L_{2,4}(x) = 1+ 5x+ 11x^2+ 15x^3$. 
It has roots $\{-1/3, (-1 \pm 2i)/5\}$ and the two complex roots are not in the region $S$.

\begin{table}
\footnotesize
\fbox{
\begin{minipage}{0.8\textwidth}
\begin{align*}
L_{2,2}(x) &= 1+ 3x\\
L_{2,3}(x) &= 1+ 4x+ 7x^2\\
L_{2,4}(x) &= 1+ 5x+ 11x^2+ 15x^3\\
L_{2,5}(x) &= 1+ 6x+ 16x^2+ 26x^3+ 31x^4\\[0.5em]
L_{3,2}(x) &= 1+ 4x+ 8x^2\\
L_{3,3}(x) &= 1+ 5x+ 15x^2+ 30x^3+ 39x^4\\
L_{3,4}(x) &= 1+ 6x+ 21x^2+ 52x^3+ 100x^4+ 148x^5+ 158x^6\\
L_{3,5}(x) &= 1+ 7x+ 28x^2+ 79x^3+ 175x^4+ 320x^5+ 490x^6+ 610x^7+ 585x^8\\[0.5em]
L_{4,2}(x) &= 1+ 5x+ 12x^2+ 20x^3\\
L_{4,3}(x) &= 1+ 6x+ 21x^2+ 53x^3+ 105x^4+ 162x^5+ 189x^6\\
L_{4,4}(x) &= 1+ 7x+ 28x^2+ 84x^3+ 203x^4+ 413x^5+ 716x^6+ 1068x^7+ 1344x^8+ 1336x^9\\[0.5em]
L_{5,2}(x) &= 1+ 6x+ 17x^2+ 32x^3+ 48x^4\\
L_{5,3}(x) &= 1+ 7x+ 28x^2+ 80x^3+ 182x^4+ 347x^5+ 561x^6+ 756x^7+ 837x^8
\end{align*}
\end{minipage}
}
\caption{Lacking polynomials $L_{m,n}(x)$ for some small $m$ and $n$. \label{tabtwo}}
\end{table}

\section*{Acknowledgments}
\noindent M.D. wishes to acknowledge the financial support from University College Dublin (OBRSS Research Support Scheme 16426) for this work.

\end{document}